%% file: BUD_ArXiv.tex
\documentclass[graybox]{svmult}

\usepackage{type1cm}        
\usepackage{makeidx}         
\usepackage{graphicx}        
\usepackage{multicol}        
\usepackage[bottom]{footmisc}

\usepackage{newtxtext}       %
\usepackage{newtxmath}       

\makeindex             

\usepackage[utf8]{inputenc}

\def\E{{\mathbb E}}
\def\R{{\mathbb R}}

\def\P{{\mathbb P}}

\def\Var{{\rm Var}}

\def\H{{\rm H}}

\def\Var{{\rm Var}}

\def\ep{\varepsilon}
\def\phi{\varphi}

\def\be{\begin{equation}}
\def\en{\end{equation}}
\def\bee{\begin{eqnarray*}}
\def\ene{\end{eqnarray*}}

\begin{document}

\title*{On rate of convergence to the Poisson law\\ of the number of cycles\\ in the generalized random graphs}
\titlerunning{Rate of convergence to the Poisson law of the numbers of cycles in GRG}
\author{ Sergey G. Bobkov, Maria A. Danshina, Vladimir V. Ulyanov}
\institute{Sergey G. Bobkov \at University of Minnesota,   Vincent Hall 228, 206 Church St SE, Minneapolis, MN 55455 USA
\at  National Research University Higher School of Economics, 101000 Moscow, Russia 
\email{bobkov@math.umn.edu}
\and Maria A. Danshina \at Moscow Center for Fundamental and Applied Mathematics, Lomonosov Moscow State University, 119991 Moscow, Russia \email{danschina.maria@yandex.ru}
\and Vladimir V. Ulyanov \at 
Lomonosov Moscow State University, 119991 Moscow, Russia \at National Research University Higher School of Economics, 101000 Moscow, Russia 
\email{vulyanov@cs.msu.ru}}
%
%
\maketitle

\abstract{Convergence of order $O(1/\sqrt{n})$ is obtained for the distance in total variation between the Poisson distribution and the distribution of the number of fixed size cycles in    generalized random graphs with random vertex weights. The weights are assumed to be   independent identically distributed random variables which have a power-law distribution. The proof is based on the Chen--Stein approach and on the derived properties of the ratio of the sum of squares of random variables and the sum of these variables. These properties can be applied to other asymptotic problems related to generalized random graphs
}

\section{Introduction}
\label{sec:1}
Complex networks attract increasing attention of researchers in various fields of science. In last years numerous network models have been proposed. Since the uncertainty and the lack of regularity in real-world networks, these models are usually random graphs. Random graphs were first defined 
in \cite{Erdos}, and independently by Gilbert in \cite{Gilbert}. The suggested models are closely related: there are $n$ isolated vertices and every possible edge occurs independently with probability $p:$ $0<p<1$. It is assumed that there are no self-loops. Later the models were generalized. A natural generalization of the Erd\H{o}s--R\'{e}nyi random graph is that the equal edge probabilities are replaced by probabilities depending on the vertex weights. Vertices with high weights are more likely to have more neighbors than vertices with small weights. Vertices with extremely high weights could act as the hubs observed in many real-world networks.

The following generalized random graph (GRG) model was first introduced by Britton et al., see \cite{Britton}. Let $V=\{1,2,..,n\}$  be the set of vertices, and $W_{i}>0$ be the weight of vertex $i$, $1 \leq i \le n.$ The edge probability of the edge between any two vertices $i$ and $j$, for $i \neq j$, is equal to 
\be \label{p_ij}
p_{ij}=\frac{W_{i}W_{j}}{L_{n}+W_{i}W_{j}}
\en
and $p_{ii}=0$ for all $i\le n.$
Here  $L_{n}=\sum_{i=1}^{n}W_{i}$ denotes the total weight of all vertices.  The weights $W_{i},$ $i = 1, 2, . . . , n$ can be taken to be deterministic or random.
If we take all $W_{i}-s$ as the same constant: $W_{i}\equiv n\lambda/(n-\lambda)$ for some $0 \le \lambda < n,$ it is easy to see that $p_{ij}=\lambda / n$ for all $1\leq i < j \leq n.$
That is, the Erd\H{o}s--R\'{e}nyi random graph with $p = \lambda/n$ is a special case of the GRG. 

There are many versions of the GRG, such as Poissonian random graph (introduced by Norros and Reittu in \cite{Norros} and studied by Bhamidi et al.\cite{Bhamidi}), rank-1 inhomogeneous random graph (see \cite{Bollobas}), random graph with given prescribed degrees (see \cite{Chung}), Chung--Lu model of heterogeneous random graph (see \cite{Chung2}) etc. The Chung--Lu model is the closest to the model of generalized random graph. Two vertices $i$ and $j$ are connected with probability $p_{ij} ={W_{i}W_{j}}/{L_{n}}$ and independently of other pairs of vertices, where  $W = (W_{1}, W_{2},...,W_{n})$ is a given sequence. It is necessary to assume that $W^{2}_{i}\leq L_{n},$ for all $i.$ Under some common conditions (see \cite{Janson}), all of the above mentioned versions of the GRG   are asymptotically equivalent, meaning that all events have asymptotically equal probabilities. The updated review on the results about these inhomogeneous random graphs see in Chapters 6 in \cite{Hofstad}.

One of the problems that arise in real networks of various nature is the spread of the virus. In \cite{Chakrabarti}, the authors proposed an approach called NLDS (nonlinear dynamic system) for modeling such processes. Consider a network of $n$ vertices represented by an undirected graph $G$. Assume an infection rate $\beta > 0$ for each connected edge that is connected to an infected vertex and a recovery rate of $\delta > 0$ for each infected individual.Define the epidemic threshold $\tau$ as a value such that 
\begin{align}
    &\beta/\delta < \tau \text{ } \Rightarrow \text{infection dies out over time} \notag \\
    & \beta/\delta > \tau \text{ }  \Rightarrow \text{infection survives and becomes an epidemic.} \notag
\end{align}
$\tau$ is related to the adjacency matrix $A$ of the graph. The matrix $A=[a_{ij}]$ is an $n\times n$ symmetric matrix defined as $a_{ij} = 1$ if vertices $i$ and $j$ are connected by an edge, and $a_{ij}$ = 0 otherwise. 
Define a walk of length $k$ in $G$ from vertex $v_{0}$ to $v_{k}$ to be an ordered sequence of vertices $(v_{0}, v_{1}, ..., v_{k}),$ with $v_{i}\in V$, such that $v_{i}$ and  $v_{i+1}$ are connected for $i = 0, 1, ..., k - 1.$ If $v_{0} = v_{k},$ then the walk is closed. A closed walk with no repeated vertices (with the exception of the first and last vertices) is called a cycle. For example, triangles, quadrangles and pentagons are cycles of length three, four, and five, respectively. In the following, the cycle will be denoted by the first $k$ vertices, without specifying the vertex $v_{k},$ which is the same as $v_{0}$: $(v_{0}, v_{1},...,v_{k-1})$.

In Theorem 1 in \cite{Chakrabarti} it has been stated that  $\tau$ is equal to $1/\lambda_{1},$ where $\lambda_{1}$ is the largest eigenvalue of the adjacency matrix $A.$ 
The following lower bound for $\lambda_{1}(A)$, was shown in \cite{Preciado}
\[
\lambda_{1}(A) \geq \frac{6\triangle+ \sqrt{36\triangle^{2}+32e^{3}/n}}{4e},
\]
where $n,$ $e,$ and $\triangle$ the number of vertices, edges and triangles in $G,$ resp. Moreover, using information about the cycle numbers of higher orders one can get more precise upper bounds for  
$\tau$. 

In \cite{Ulyanov}    the central limit theorems were proved for the total number of edges in GRG.  
There are also many results on asymptotic properties of the number of triangles in homogeneous cases. For example, for the Erd\H{o}s--R\'{e}nyi random graph, the upper tails for the distribution of the triangle number had been studied 
in \cite{Goldstein}, \cite{Demarco}, \cite{Janson2}, \cite{Kim}. Recently, 
  in \cite{Liu2020} it was shown for GRG model that asymptotic distribution of the triangle number converges to a Poisson distribution 
  under strong assumption that the vertex weights are bounded random variables.

A lot of real-world networks such as social or computer networks in Internet, see e.g. \cite{Faloutsos}, follow a so-called scale-free graph model, see Ch.1 in \cite{Hofstad}. In Ch.6 in  \cite{Hofstad}  it was shown that when the vertex weights have   approximately a power-law distribution, the GRG model leads to scale-free random graph.

In the present paper, we prove not only the convergence but we get the convergence rate of order $O(1/\sqrt{n})$ for the distance in total variation between the Poisson distribution and the distribution of the number of fixed size cycles in GRG with random vertex weights. The weights are assumed to be  independent identically distributed random variables which have a power-law distribution. The proof is based on the Chen--Stein approach and on the derived properties of the ratio of the sum of squares of random variables and the sum of these variables. These properties can be applied to other asymptotic problems related to GRG.

The main results are formulated in Section \ref{sec:2}. For their proofs, see  Section \ref{sec:4}. Section \ref{sec:3} contains auxiliary lemmas, some of which  are of independent interest.

\section {Main results} 
\label{sec:2}
Let $\{1, 2,..., n\}$ be the set of vertices, and $W_{i}$ be a weight of vertex $i:\text{ }1\leq i \leq n.$ The probability of the edge between vertices $i$ and $j$ is defined in  \eqref{p_ij}.
Let $W_{i},$ $i=1,2,...,n,$ be independent identically distributed random variables distributed as a random variable $W$. For $k\geq 3$, denote by $I(k)$ the set of potential cycles of length $k$. We have that the number of elements in $I(k)$ is equal to $(n)_{k}/(2k),$ where $(n)_{k}=n(n-1)...(n-k+1)$ is the number of ways to select k distinct vertices in order, and the factor $1/(2k)$ appears since, for $k>2$, a permutation of $k$ vertices corresponds to a choice of a cycle in $I(k)$ together with a choice of any of two orientations and $k$ starting points. For example, all six cycles $\{1,3,4\},$ $\{3,4,1\},$ $\{4,1,3\}$, $\{4,3,1\},$ $\{1,4,3\},$ $\{3,1,4\}$ are, in fact, one cycle of length $3$. For $\alpha \in I(k),$ let $Y_{\alpha}$ be the indicator that $\alpha$ occurs as a cycle in GRG. 
For example, $\P (Y_{\{1,3,4\}}=1)=p_{13}p_{34}p_{41}$. 


For any integer-valued non-negative random variables $Y$ and $Z$, denote the total variation distance between their  distributions $\mathscr{L}(Y)$ and $\mathscr{L}(Z)$ 
by 
\be \label{1}
\parallel \mathscr{L}(Y)-\mathscr{L}(Z)  \parallel \equiv sup_{\parallel h \parallel= 1} |\E h(Y) -\E h(Z) |,
\en
where $h$ is any real function defined on $\{0,1,2,...\}$ and $\parallel h \parallel \equiv \sup_{m \geq 0} |h(m)|.$

For  $k\geq 3,$ put $S_{n}(k)=\sum_{\alpha \in I(k)}Y_{\alpha}$, that is $S_{n}(k)$ is the number of cycles of length $k.$ Let $Z_{k}$ be a random variable having Poisson distribution with parameter $\lambda(k)=(\E W^{2}/\E W)^{k}/(2k).$
\begin{theorem} \label{theorem1}
For any $k\geq 3,$ one has 
\be \label{2}
\parallel \mathscr{L}(S_{n}(k))-\mathscr{L}(Z_{k})  \parallel = O(n^{-1/2}),
\en
provided that 
\be \label{asumpW}
\P(W>x)=o(x^{-2k-1}), \ \ {\rm as} \ \ x\rightarrow +\infty. 
\en
\end{theorem}
\begin{remark}
Relation \eqref{2} holds under condition that $W$ has power-law distribution. The condition on the tail behaviour of the distribution of $W$ can be replaced by stronger moment condition: the finiteness of expectation $\E W^{2k+1}.$
\end{remark}
\begin{remark}
Recently in \cite{Liu2020} it was proved by method of moments the convergence in distribution of the number of triangles   $S_{n}(3)$ in a generalized
random graphs 
to the Poisson random variable $Z_3$ under assumption that the vertex weights $W_i$-s are bounded random variables. In Theorem \ref{theorem1} we have used the Chen--Stein approach, see e.g. \cite{Goldstein89} and \cite{Goldstein}. This allows us not only to extend the convergence result to cycles of any fixed length $k$ but also to get the rate of convergence. Moreover, we replace the assumption about the boundness of $W_i$-s with the condition that $W_i$ has a power-law distribution. As we noted in the Introduction, this condition better matches real-world networks.
\end{remark}
Figures \ref{figLeft} and \ref{figRight} illustrate the results of Theorem \ref{theorem1}, with the example of the number of triangles and quadrilaterals distributions.
\begin{figure}[bh]
    \begin{multicols}{2} 
        \hfill
        \includegraphics[width=55mm]{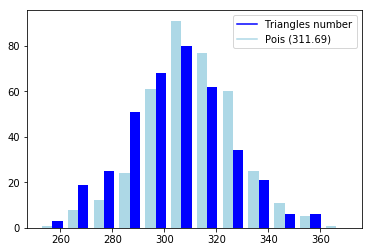}
        \hfill
        \caption{Histogram of the  number of triangles in GRG  with $2000$ vertices. The vertex weights $W_{i}=scale*Y+loc,$ where $loc=1,$ $scale=10$ and $Y \thicksim Pareto$ $(9.5)$ for all $i \leq 2000.$ The number of realizations is 400.}
        \label{figLeft}
        \hfill
        \includegraphics[width=42mm]{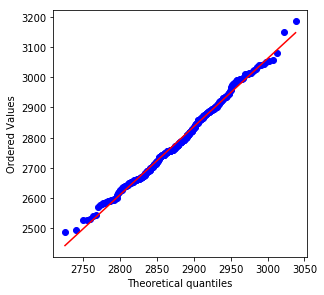}
        \hfill
        \hfill
        \caption{Q-Q plot for the  number of quadrilaterals in GRG  with $2000$ vertices and the Poisson variable  $Pois(2880.16)$. The vertex weights $W_{i}=scale*Y+loc,$ where $loc=1,$ $scale=10$ and $Y \thicksim Pareto$ $(9.5)$ for all $i \leq 2000.$ The number of realizations is 400.}
        \label{figRight}
    \end{multicols}
\end{figure}

The next results are not directly connected with number of cycles in GRG. They are an important part of the proof of the Theorem \ref{theorem1}. 
At the same time the results are of independent interest. They describe the asymptotic properties of ratio of a sum of the squares of $n$ i.i.d. random variables and a sum of these random variables. These properties can be applied to other asymptotic problems related to GRG.

Given i.i.d. positive random variables $X, X_1, \dots , X_n$,   define the statistics
$$
T_n \, = \, \frac{X_1^2 + \dots + X_n^2}{X_1 + \dots + X_n}.
$$
Assume that $X$ has a finite second moment, so that, by the law of large
numbers, with probability one
$$
\lim_{n \rightarrow \infty} T_n^p \, = \, \Big(\frac{\E X^2}{\E X}\Big)^p
$$
for any $p \geq 1$. Here we describe the tail-type and moment-type conditions which ensure that 
this convergence also holds on average. 
\begin{theorem} \label{proposition:1.2}
Given an integer $p \geq 2$, the convergence
\be \label{1.1}
\lim_{n \rightarrow \infty} \E T_n^p \, = \, (\E X^2/\E X)^p
\en
is equivalent to the tail condition
\be \label{3.4}
\P\{X \geq x\} = o(x^{-p-1}) \ \ {\rm as} \ \ x \rightarrow \infty.
\en
Moreover, if $\P\{X \geq x\} = O(x^{-p-{3}/{2}})$ as 
$x \rightarrow \infty$, then
\be \label{1.2}
\E T_n^p - (\E X^2/\E X)^p \, = \, O(n^{-1/2})
\en
The finiteness of the moment $\E X^{p+1}$ is sufficient for \eqref{1.1}
to hold, while the finiteness of the moments $\E X^q$ is necessary for
any real value $1 \leq q < p+1$.
\end{theorem}

Let $M_{n}=\max_{1\le i\le n} \, X_{i}.$ For $p \geq 2$, define
\be \label{R_n}
R_{n}^{(p)}=T_{n}^{p}M_{n}^{2}/(X_{1}+X_{2}+...+X_{n}).
\en

By the law of large numbers, $R_n^{(p)} \rightarrow 0$ as $n \rightarrow \infty$ a.s.,
under mild moment assumptions. The next theorem gives the order of convergence of $\E R_n^{(p)}$ to zero under tail-type and moment-type  conditions.   
\begin{theorem} \label{proposition:3}
 Given an integer $p \geq 2$, if $\P(X\ge x)=O(x^{-p-7/2})$ as $x\rightarrow +\infty,$ then  
\be \label{th-3}
\E R_{n}^{(p)} \, = \, O(n^{-1/2}).
\en
When $p>8$ and $\E X^{p+4}$ is finite, the rate can be improved to
\be \label{th-31}
\E R_n^{(p)} = O(n^{-{(p-2)}/{(p + 4)}}).
\en
Moreover, if
$\E\, e^{\ep X} < \infty$ for some $\ep > 0$, then
\be \label{8.2}
\E R_n^{(p)} = O\Big(\frac{(\log n)^2}{n}\Big).
\en
\end{theorem}
 

\section{Auxiliary lemmas}
\label{sec:3}

\begin{lemma} \label{lemma:1.3}
Let $S_n = \eta_1 + \dots + \eta_n$ be the sum of independent 
random variables $\eta_k \geq 0$ with finite second moment, such that $\E S_n = n$ 
and $\Var(S_n) = \sigma^2 n$. Then, for any $0 < \lambda < 1$, one has
\be \label{exp}
\P\{S_n \leq \lambda n\} \leq \exp\bigg\{- \frac{(1-\lambda)^2}{2\,
\big[\sigma^2 + \max_k\, (\E \eta_k)^2\big]}\,n\bigg\}.
\en
\end{lemma} 
\begin{proof} 
We use here the standard arguments. Fix a parameter $t>0$.
We have
$$
\E\, e^{-t S_n} \geq e^{-\lambda t n}\, \P\{S_n \leq \lambda n\}. 
$$
Every function $u_k(t) = \E\, e^{-t \xi_k}$ is positive, 
convex, and admits Taylor's expansion near zero up to the quadratic form, 
which implies that
$$
u_k(t) \leq 1 - t\, \E \xi_k + \frac{t^2}{2}\, \E \xi_k^2
\leq \exp\Big\{ - t\, \E \xi_k + \frac{t^2}{2}\, \E \xi_k^2\Big\}.
$$
Multiplying these inequalities, we get
$$
\E\, e^{-t S_n} \leq \exp\Big\{-t n + \frac{bt^2}{2}\Big\}, \quad 
b = \sum_{k=1}^n  \E \xi_k^2.
$$
The two bounds yield
$$
\P\{S_n \leq \lambda n\} \, \leq \, \exp\Big\{-(1 - \lambda)n t + bt^2/2\Big\},
$$
and after optimization over $t$ (in fact, $t = \frac{1-\lambda}{b}\,n$), 
we arrive at the exponential bound
$$
\P\{S_n \leq \lambda n\} \leq 
\exp\Big\{- \frac{(1-\lambda)^2}{2b}\,n^2\Big\}.
$$
Note that
$$
b = \Var(S_n) + \sum_{k=1}^n\, (\E \xi_k)^2 \leq 
\Big(\sigma^2 + \max_k\, (\E \xi_k)^2\Big)\,n,
$$
and \eqref{exp} follows.
\end{proof}





For further lemmas we need additional notation.

Denote by $F(x) = \P\{X \leq x\}$ ($x \in \R$) the distribution function 
of the random variable $X$ and put
\be \label{eps}
\ep_q(x) = x^q\,(1-F(x)), \quad x \geq 0, \ q>0. \nonumber
\en

Raising the sum $U_n = X_1^2 + \dots + X_n^2$ to the power $p$ with $n \geq 2p$, 
we have
\be \label{3.1}
U_n^p \, = \, \sum X_{i_1}^2 \dots X_{i_p}^2,
\en
where the summation is performed over all collections of numbers
$i_1,\dots,i_p \in \{1,\dots,n\}$. 
 For $r = 1,\dots,p$, 
we denoted by $\mathcal C(p,r)$ the collection  of all tuples
$\gamma = (\gamma_1,\dots,\gamma_r)$ of positive integers such that 
$\gamma_1 + \dots + \gamma_r = p$. For any $\gamma \in \mathcal C(p,r)$, 
there are $n(n-1)\dots (n-r+1)$ sequences $X_{i_1},\dots, X_{i_p}$ 
with $r$ distinct terms that are repeated $\gamma_1,\dots,\gamma_r$ times, resp. 
Therefore, by the i.i.d. assumption,
\be \label{3.2}
\E T_n^p \ = \ \sum_{r=1}^p \frac{n(n-1)\dots (n-r+1)}{n^p}
\sum_{\gamma \in \mathcal C(p,r)} \E \xi_n(\gamma),
\en
where
\be \label{xi}
\xi_n(\gamma) = {X_1^{2\gamma_1} \dots 
X_r^{2\gamma_r}}/{(\frac{1}{n}\,S_r + \frac{1}{n}\, S_{n,r})^p}\nonumber
\en
and 
$$
S_r = X_1 + \dots + X_r, \quad S_{n,r} = X_{r+1} + \dots + X_n.
$$ 
\long\gdef\COMMENT#100{
So, consider the collection $\mathcal C(p,r)$ of all tuples
$\gamma = (\gamma_1,\dots,\gamma_r)$ of positive integers  such that
$\gamma_1 + \dots + \gamma_r = p$. Let us gather together the sequences 
$X_{i_1},\dots, X_{i_p}$ in \eqref{4.1} in which there are $r$ terms that are 
repeated $\gamma_1,\dots,\gamma_r$ times. 

For example, the tuple $\gamma = (1,\dots,1)$ of length $r=p$ corresponds to the sequences $X_{i_1},\dots, X_{i_p}$ with distinct indexes, and there are exactly $n(n-1)\dots (n-p+1)$ such sequences. The tuple $\gamma = (p)$ of length $r=1$ corresponds to the sequences $X_{i_1},\dots, X_{i_p}$ with equal indexes, and there are exactly $n$ such sequences. More generally, for any fixed $\gamma \in \mathcal C(p,r)$, there are exactly $n(n-1)\dots (n-r+1)$ sequences with a required property.

Put
$$
S_n = X_1 + \dots + X_n, \quad S_{n,r} = X_{r+1} + \dots + X_n.
$$ 
Since $T_n^p = \frac{U_n^p}{S_n^p}$, by the i.i.d assumption,
we have the representation generalizing \eqref{3.1}, namely
}
\long\gdef\COMMENT#101{
For the particular collection $\gamma = (p)$ with $r=1$ we have
$\xi_n(\gamma) = \frac{X_1^{2p}}{(\frac{1}{n}\,X_1 + \frac{1}{n}\, S_{n,1})^p}$,
\bee
\E \xi_n(\gamma) 
 & \geq &
\E_{X_1}\,\frac{X_1^{2p}}{(\frac{1}{n}\,X_1 + \frac{1}{n}\,\E_{S_{n,1}}\, S_{n,1})^p}\\
 & = &
\E\,\frac{X^{2p}}{(\frac{1}{n}\,X + \frac{n-1}{n})^p} \ \geq \
2^{-p}\,n^p \ \E X^p\, 1_{\{X \geq n\}},
\ene
where we applied Jensen's inequality.
In view of \eqref{4.2}, for the boundedness of the sequence $\E T_n^p$
it is therefore necessary that
\be \label{3.3}
\E X^p\, 1_{\{X \geq n\}} = o(1/n)
\en
as $n \rightarrow \infty$. In particular, the moment $\E X^p$ has to be finite.

The relation \eqref{3.3} may be simplified in terms of the tails
of the distribution of $X$. Indeed,
$$
\E\,X^p\,1_{\{X \geq n\}} \, \geq \, n^p\,\P\{X \geq n\},
$$
so that the property
\be \label{3.4}
\P\{X \geq n\} = o(1/n^{p+1}) \qquad (n \rightarrow \infty)
\en
is necessary for \eqref{3.3}. On the other hand, from \eqref{3.4} it follows that
$\ep_{p+1}(x) \rightarrow 0$ as $x \rightarrow \infty$.
Hence, assuming without loss of generality that $x = n$ is the point of 
continuity of $F$, we get
\bee
\E\, X^p\,1_{\{X > n\}}
 & = &
\int_n^\infty x^p\,dF(x) \ = \
n^p\,(1 - F(n)) + p\,\int_n^\infty x^{p-1}\,(1-F(x))\,dx \\
 & = &
o(1/n) + p\,\int_n^\infty \frac{\ep(x)}{x^2}\,dx \ \leq \ o(1/n) + 
p\,\sup_{x > n} \, \ep(x)\,\frac{1}{n} \ = \ o(1/n).
\ene
Thus, the tail condition \eqref{3.4} is necessary for the convergence 
$\E T_n^p \rightarrow (\E X^2)^p$ as stated in Proposition \ref{proposition:1.2}
(and actually for the boundedness of the $p$-th moments of $T_n$).

\vskip10mm
}

%

In the following lemmas, without loss of generality, let $\E X = 1$.

\begin{lemma} \label{lemma:3.1}
For the boundedness of the sequence $\E T_n^p$ it is 
necessary that the moment $\E X^p$ be finite. Moreover, for the particular 
collection $\gamma = (p)$ with $r=1$, we have 
\be \label{short}
\E \xi_n(\gamma) \, \geq \, 2^{-p}\,n^p \ \E X^p\, 1_{\{X \geq n\}}.
\en
\end{lemma}
\begin{proof}
Since
$\xi_n(\gamma) = {X_1^{2p}}/{(\frac{1}{n}\,X_1 + \frac{1}{n}\, S_{n,1})^p}$,
applying Jensen's inequality, we get
\bee
\E \xi_n(\gamma) 
 & \geq &
\E_{X_1}\,\frac{X_1^{2p}}{(\frac{1}{n}\,X_1 + \frac{1}{n}\,\E_{S_{n,1}}\, S_{n,1})^p}\\
 & = &
\E\,\frac{X^{2p}}{(\frac{1}{n}\,X + \frac{n-1}{n})^p} \ \geq \
2^{-p}\,n^p \ \E X^p\, 1_{\{X \geq n\}}.
\ene
\end{proof}
In the sequel, we use the events
\be \label{AB}
A_{n,r} = \Big\{S_{n,r} \leq \frac{n-r}{2}\Big\} \quad {\rm and} \quad
B_{n,r} = \Big\{S_{n,r} > \frac{n-r}{2}\Big\}. 
\en
By Lemma \ref{lemma:1.3}, whenever $n \geq 2p$, for some constant $c>0$ independent of $n$:
\be \label{5.2}
\P(A_{n,r}) \leq e^{-c(n-r)} \leq e^{-cn/2}.
\en
\begin{lemma} \label{lemma:3.2}
If $\E X^p$ is finite, then
$\E \xi_n \rightarrow (\E X^2)^p$ as $n \rightarrow \infty$, where
\be \label{xi_n}
\xi_n \, = \, {X_1^2 \dots 
X_p^2}/{(\frac{1}{n}\,S_p + \frac{1}{n}\, S_{n,p})^p}.
\en
\end{lemma}
\begin{proof}
Using $X_1 \dots X_p \leq S_p^p$, we have
$
\xi_n \, \leq \, 
{S_p^{2p}}/{(\frac{1}{n}\,S_n)^p} \, \leq \,
n^p\,S_p^p.
$
Hence
\be \label{long}
\E\,\xi_n\,1_{A_{n,p}} \, \leq \, 
n^p\ \E\,S_p^p\ \P(A_{n,p}) \, = \,
o(e^{-cn})
\en
for some constant $c>0$ independent of $n$. Here, we applied \eqref{5.2} with
$r=p$ and Lemma \ref{lemma:3.1} which ensures that $\E\,S_p^p < \infty$. Further, 
%
$
\xi_n \,1_{B_{n,p}}
\, \leq \,
2^p X_1^2 \dots X_p^2.
$
Hence, the random variables $\xi_n\,1_{B_{n,p}}$ have an an integrable 
majorant. Since also $\xi_n \rightarrow X_1^2 \dots X_p^2$ (the law of large 
numbers) and $1_{B_{n,p}} \rightarrow 1$ a.s. (implied by \eqref{5.2}), 
one may apply the Lebesgue dominated convergence theorem, which gives 
$
\E \xi_n 1_{B_n} \rightarrow (\E X^2)^p.
$
Together with \eqref{long}, we get $\E \xi_n \rightarrow (\E X^2)^p$.
\end{proof}
\begin{lemma} \label{lemma:4.1}
 If the moment $\E X^p$ is finite, then for any 
$\gamma = (\gamma_1,\dots,\gamma_r) \in \mathcal C(p,r)$,
$$
\E\,\xi_n(\gamma) \, = \, 4^p\,\E\,\frac{X_1^{2\gamma_1} \dots 
X_r^{2\gamma_r}}{(\frac{1}{n}\,S_r + 1)^p} + o(1).
$$
\end{lemma}
\begin{proof}
 Using an elementary bound
$
X_1^{2\gamma_1} \dots X_r^{2\gamma_r} \leq 
(X_1 + \dots + X_r)^{2\gamma_1 + \dots + 2\gamma_r} = S_r^{2p}
$
and applying Jensen's inequality, we see that
$
\xi_n(\gamma) \leq n^p\, S_r^p \leq n^p\,r^{p-1}\,(X_1^p + \dots + X_r^p).
$
Hence
\be \label{4.1}
\E \xi_n(\gamma)\,1_{A_{n,r}} \leq
n^p\,r^{p-1} \sum_{k=1}^r \E X_k^p\,1_{A_{n,r}} = 
n^{p}\,r^{p}\, \E X^p\,\P(A_{n,r}) \ = \ o(e^{-c' n}).
\en

On the other hand, on the set $B_{n,r}$ there is a pointwise bound
\be \label{4.1.1}
\xi_n(\gamma) \,1_{B_{n,r}} \leq \, \frac{X_1^{2\gamma_1} \dots 
X_r^{2\gamma_r}}{(\frac{1}{n}\,S_r + \frac{n-r}{2n})^p} \, \leq \, 4^p\,
\frac{X_1^{2\gamma_1} \dots 
X_r^{2\gamma_r}}{(\frac{1}{n}\,S_r + 1)^p}.
\en
\end{proof} 

Our task is reduced to the estimation of the expectation
on the right-hand side of \eqref{4.1.1}. Let us first consider the
{shortest collection $\gamma = (p)$ of length $r=1$.

\vskip5mm
\begin{lemma} \label{lemma:4.2}
 Under the condition \eqref{3.4},
\be \label{4.2}
\E\, \frac{X_1^{2p}}{(\frac{1}{n}\,X_1 + 1)^p} = o(n^{p-1}).
\en
In addition, if $\P\{X \geq x\} = O(x^{-q})$ for some real value $q$ in
the interval $p < q < 2p$, then
\be \label{4.3}
\E\, \frac{X_1^{2p}}{(\frac{1}{n}\,X_1 + 1)^p} = O(n^{2p-q}).
\en
\end{lemma} 
\vskip2mm
\begin{proof}
We have
\begin{eqnarray} 
\label{L3-1}
\E\, \frac{X_1^{2p}}{(\frac{1}{n}\,X_1 + 1)^p} 
 & = &
\E\, \frac{X_1^{2p}}{(\frac{1}{n}\,X_1 + 1)^p}\,1_{\{X_1 \geq n\}} +
\E\, \frac{X_1^{2p}}{(\frac{1}{n}\,X_1 + 1)^p}\,1_{\{X_1 < n\}} \nonumber\\
 & \leq &
n^p\, \E\,X^p\,1_{\{X \geq n\}} + \E\, X^{2p}\,1_{\{X < n\}}.
\nonumber
\end{eqnarray}

In view of \eqref{3.4}, to derive \eqref{4.2}, it remains to bound the last expectation 
by $o(n^{p-1})$. Integrating by parts and assuming that
$x = n$ is the point of continuity of $F(x)$, we have using $\ep_{p+1}(x) \rightarrow 0$ as $x \rightarrow \infty$,
\begin{eqnarray} 
\label{by_part}
\E\, X^{2p}\,1_{\{X < n\}}
 & = &
-n^{2p}\,(1 - F(n)) + 2p \int_0^n x^{2p-1}\,(1-F(x))\,dx \nonumber\\
 & \leq &
2p \int_0^n x^{p-2}\,\ep_{p+1}(x)\,dx \ = \ o(n^{p-1}),
\end{eqnarray}

For the second assertion \eqref{4.3}, we similarly have
\bee
\E\, X^{2p}\,1_{\{X < n\}} \, &\leq &\,
2p \int_0^n x^{2p-1-q}\,\ep_q(x)\,dx \ = \ O(n^{2p-q}),\\
\E\, X^p\,1_{\{X \geq n\}}
&= &
O(n^{p-q}) + p \int_n^\infty x^{p-q-1}\,\ep_q(x)\,dx \ = \ O(n^{p-q}).
\ene
\long\gdef\COMMENT#106{
It follows From \eqref{3.4} that
$\ep_{p+1}(x) \rightarrow 0$ as $x \rightarrow \infty$.
Hence, assuming without loss of generality that $x = n$ is the point of 
continuity of $F$, we get
\begin{eqnarray} \label{tail}
\E\, X^p\,1_{\{X > n\}}
 & = &
\int_n^\infty x^p\,dF(x) \ = \
n^p\,(1 - F(n)) + p\,\int_n^\infty x^{p-1}\,(1-F(x))\,dx 
\nonumber\\
 & = &
o(1/n) + p\,\int_n^\infty \frac{\ep(x)}{x^2}\,dx \\ 
& = &
o(1/n) + 
p\,\sup_{x > n} \, \ep(x)\,\frac{1}{n} \ = \ o(1/n).\nonumber
\end{eqnarray}

In view of \eqref{tail}, to derive \eqref{4.2}, it remains to bound the last expectation in \eqref{L3-1} 
by $o(n^{p-1})$. Integrating by parts and assuming that
$x = n$ is the point of continuity of $F(x)$, we have
\begin{eqnarray} 
\label{by_part}
\E\, X^{2p}\,1_{\{X < n\}}
 & = &
\int_0^n x^{2p}\,dF(x) 
=
-n^{2p}\,(1 - F(n)) + 2p \int_0^n x^{2p-1}\,(1-F(x))\,dx \nonumber\\
 & \leq &
2p \int_0^n x^{p-2}\,\ep_{p+1}(x)\,dx \ = \ o(n^{p-1}),
\end{eqnarray}
where we used that $\ep_{p+1}(x) \rightarrow 0$ as $x \rightarrow \infty$.

For the second assertion \eqref{4.3}, we similarly have
$$
\E\, X^{2p}\,1_{\{X < n\}} \, \leq \,
2p \int_0^n x^{2p-1-q}\,\ep_q(x)\,dx \ = \ O(n^{2p-q}).
$$
In addition,
\bee
\E\, X^p\,1_{\{X \geq n\}}
 & = &
n^p\,(1 - F(n)) + p \int_n^\infty x^{p-1}\,(1-F(x))\,dx \\
 & = &
O(n^{p-q}) + p \int_n^\infty x^{p-q-1}\,\ep_q(x)\,dx \ = \ O(n^{p-q}).
\ene
}
\end{proof}

\begin{lemma} \label{lemma:4.3}
Let $\gamma = (\gamma_1,\dots,\gamma_r) \in \mathcal C(p,r)$,
$2 \leq r \leq p-1$. Under \eqref{3.4}, we have
\be \label{4.4}
\E\,\frac{X_1^{2\gamma_1} \dots X_r^{2\gamma_r}}{(\frac{1}{n}\,S_r + 1)^p}
 = o(n^{p-r-1} \log n).
\en
\end{lemma} 
\begin{proof}
If all $\gamma_i \leq {p}/{2}$, there is nothing to prove,
since then
$$
\E\,\frac{X_1^{2\gamma_1} \dots X_r^{2\gamma_r}}{(\frac{1}{n}\,S_r + 1)^p}
\, \leq \, \E\,X_1^{2\gamma_1} \dots \E\,X_r^{2\gamma_r} \, \leq \,
(\E X^p)^r.
$$

In the other case, suppose for definiteness that $\gamma_1$ is the largest 
number among all $\gamma_i$'s. Necessarily $\gamma_1 >  {p}/{2}$ and
$\gamma_i < {p}/{2}$ for all $i \geq 2$. Since $S_r < n$ implies
$X_1 < n$, we similarly have
$$
\E\,\frac{X_1^{2\gamma_1} \dots X_r^{2\gamma_r}}{(\frac{1}{n}\,S_r + 1)^p}\
1_{\{S_r < n\}} \, \leq \, (\E X^p)^{r-1}\,
\E\,X^{2\gamma_1}\,1_{\{X < n\}}.
$$
To bound the last expectation, note that
$r - 1 \leq \gamma_2 + \dots + \gamma_r <  {p}/{2}$, so that $p \geq 2r-1$. 
Hence, if $x = n$ is the point of continuity of $F(x)$, similarly to \eqref{by_part} we get
\be  \label{4.5}
\E\, X^{2\gamma_1}\,1_{\{X < n\}}
 \leq
2\gamma_1 \int_0^n x^{2\gamma_1 - p -2}\,\ep_{p+1}(x)\,dx.
\en
But since $\gamma_1 \leq p-r+1$,
\be \label{4.6}
\int_1^n x^{2\gamma_1 - p -2}\,\ep_{p+1}(x)\,dx \, \leq \, 
\int_1^n x^{p - 2r}\,\ep_{p+1}(x)\,dx \, = \, o(n^{p - 2r + 1}), 
\en
if  $p \geq 2r$ or $p \leq 2r - 2,$
which is even stronger than the rate $o(n^{p-r-1})$.
In the remaining case $p = 2r-1$, the last integral is $o(\log n)$.
This proves \eqref{4.4} for the part of the expectation restricted to the set $S_r<n$,
that is,
\be \label{4.7}
\E\,\frac{X_1^{2\gamma_1} \dots X_r^{2\gamma_r}}{(\frac{1}{n}\,S_r + 1)^p}\ 
1_{\{S_r < n\}} \, = \, o(n^{p-r-1} \log n).
\en
Note that the logarithmic term cannot be removed in the special situation where 
$p=3$, $r=2$, $\gamma_1 = 2$, $\gamma_2 = 1$, in which case the last integral 
in \eqref{4.6} becomes $\int_1^n x^{-1}\,\ep_4(x)\,dx$.

Turning to the expectation over the complementary set $S_r \geq n$,
introduce the events 
$$
\Omega_i = \Big\{X_i \geq \max_{j \neq i} X_j\Big\}, \quad i = 1,\dots,r.
$$
On every such set, $X_i \leq S_r \leq rX_i$. In particular,
$S_r \geq n$ implies $X_i \geq n/r$. Hence, together with \eqref{4.7}, 
\eqref{4.4} would follow from the stronger assertion
\be\label{4.8}
\E\,\frac{X_1^{2\gamma_1} \dots X_r^{2\gamma_r}}{X_i^p}\,
1_{\{X_i \geq n\} \cap \Omega_i} \, = \, o(n^{-r-1})
\en
with an arbitrary index $1 \leq i \leq r$.

{\it Case 1.}
$i \geq 2$. If we fix any values $X_1 = x_1$ and $X_i = x_i$,
then the expectation with respect to $X_j$, $j \neq i$, in \eqref{4.8} will yield 
a bounded quantity (since the $p$-moment is finite). Hence \eqref{4.8} is simplified to
\be \label{4.9}
\E\,X_1^{2\gamma_1} X_i^{2\gamma_i - p}\,
1_{\{X_i \geq n\} \cap \{X_i \geq X_1\}} \, = \, o(n^{-r-1}).
\en
Here, the expectation over $X_1$ may be carried out and estimated similarly to 
\eqref{4.5}, by replacing $n$ with $x_i$. Namely,
$$
\E\,X_1^{2\gamma_1} 1_{\{X_1 \leq x_i\}} \, \leq \, 
2\gamma_1 \int_0^{x_i} x^{2\gamma_1 - p - 2}\,\ep_{p+1}(x)\,dx \, = \, 
\delta(x_i)\,x_i^{2\gamma_1 - p}
$$
with some $\delta(x_i) \rightarrow 0$ as $x_i \rightarrow \infty$
(this assertion may be strengthened when $2\gamma_1 - p = 1$).
Hence, the expectation in \eqref{4.9} is bounded by
\bee 
\E\,X_i^{2\gamma_i + 2\gamma_1 - 2p}\,\delta(X_i)\,1_{\{X_i \geq n\}}
 & \leq &
\delta_n \ \E\,X_i^{2\gamma_i + 2\gamma_1 - 2p}\,1_{\{X_i \geq n\}} \\
 & & \hskip-30mm = \ 
\delta_n\,n^{2\gamma_i + 2\gamma_1 - 2p}\,(1-F(n)) +
c_i \delta_n \int_n^\infty x^{2\gamma_i + 2\gamma_1 - 2p - 1}\,(1-F(x))\,dx \\
 & & \hskip-30mm = \ 
o(n^{2\gamma_i + 2\gamma_1 - 3p - 1}) +
c_i \delta_n \int_n^\infty x^{2\gamma_i + 2\gamma_1 - 3p - 2}\,\ep_{p+1}(x)\,dx \\
 & & \hskip-30mm = \ 
o(n^{2\gamma_i + 2\gamma_1 - 3p - 1}),
\ene
where $\delta_n = \sup_{x \geq n} \delta(x) \rightarrow 0$.
To obtain \eqref{4.9}, it remains to check that
$2\gamma_i + 2\gamma_1 - 3p - 1 \leq -r-1$. And indeed, since
$
p = \gamma_i + \gamma_1 + \sum_{j \neq i,1} \gamma_j \geq \gamma_i + \gamma_1 + (r-2),
$
the desired relation would follow from
$2\,(p - (r-2)) - 3p - 1 \leq -r-1$, that is, $p+r \geq 4$ (which is true).

{\it Case 2.}
$i = 1$. If we fix any value $X_1 = x_1$,
the expectation with respect to $X_j$, $j \neq 1$, will yield a bounded
quantity (since the $p$-moment is finite). Hence \eqref{4.8} is simplified to
\be \label{4.10}
\E\,X^{2\gamma_1 - p}\, 1_{\{X \geq n\}} = o(n^{-r-1}).\nonumber
\en
Here, the expectation may be estimated similarly.
Namely,
\bee
\E\,X^{2\gamma_1 - p}\, 1_{\{X \geq n\}} 
 & = &
\int_n^\infty x^{2\gamma_1 - p}\,dF(x) \\
 & = &
o(n^{2\gamma_1 - 2p-1}) +
\int_n^\infty x^{2\gamma_1 - 2p - 2}\,\ep_{p+1}(x)\,dx \ = \
o(n^{2\gamma_1 - 2p - 1}).
\ene
It remains to see that $2\gamma_1 - 2p - 1 \leq -r-1$.
Again, since $\gamma_1 \leq p - (r-1)$, the latter would follow from
$2(p - r + 1) - 2p - 1 \leq -r-1$, which is the same as $r \geq 2$.
\end{proof}

\long\gdef\COMMENT#103{
\vskip10mm
\section{Proof of Proposition \ref{proposition:1.2}}
\label{sec:9}

\vskip2mm
\noindent
For the convergence part, we apply Lemmas \ref{lemma:4.1}-\ref{lemma:4.3}, which imply that 
$\E \xi_n(\gamma) = o(n^{p-r})$ for any collection
$\gamma = (\gamma_1,\dots,\gamma_r)$ with $r<p$, and it remains to consider
the longest tuple $\tilde \gamma = (1,\dots,1)$ of length $r=p$, in which case
$$
\xi_n \equiv \xi_n(\tilde \gamma) \, = \, \frac{X_1^2 \dots 
X_p^2}{(\frac{1}{n}\,S_n)^p}.
$$
According to the representation \eqref{3.2}, we thus obtain the characterization
\be \label{5.1}
\E T_n^p \rightarrow (\E X^2)^p \ \Longleftrightarrow \ \E \xi_n = 
\E\,\frac{X_1^2 \dots X_p^2}{(\frac{1}{n}\,S_n)^p} \rightarrow
(\E X^2)^p.
\en

Let us use the same notations as before.
As we have already noticed in Section \ref{sec:7}, 
\be \label{5.2}
\P(A_{n,p}) \, \leq \, e^{-cn}
\en
and, for any $\gamma \in \mathcal C(p,r)$,
\be \label{5.3}
\E\,\xi_n(\gamma)\,1_{A_{n,r}} \, = \, o(e^{-cn})
\en
for some constant $c>0$, cf. \eqref{4.1}. On the set $B_{n,p}$ 
there is pointwise upper bound 
$$
\xi_n \, = \, \frac{X_1^2 \dots 
X_p^2}{(\frac{1}{n}\,S_p + \frac{1}{n}\, S_{n,p})^p} \, \leq \,
2^p X_1^2 \dots X_p^2.
$$
Hence, the random variables $\xi_n\,1_{B_{n,p}}$ have an an integrable 
majorant. Since also $\xi_n \rightarrow X_1^2 \dots X_p^2$ 
(the law of large numbers) and $1_{B_{n,p}} \rightarrow 1$ a.s. (implied by \eqref{5.2}), 
one may apply the Lebesgue dominated convergence theorem, which gives 
$$
\E \xi_n 1_{B_n} \rightarrow (\E X^2)^p.
$$
Together with \eqref{5.3}, we get $\E \xi_n \rightarrow (\E X^2)^p$,
and \eqref{5.1} is proved.

Turning to the rate of convergence, firs note that by
Lemma \ref{lemma:4.1} and \ref{lemma:4.3}, for any $\gamma \in \mathcal C(p,r)$ with $2 \leq r \leq p-1$,
\be \label{5.4}
\frac{n(n-1)\dots (n-r+1)}{n^p}\ \E \xi_n(\gamma) \, = \, 
o\Big(\frac{\log n}{n}\Big)
\en
For the shortest tuple $\gamma = (p)$ with $r=1$, we apply
Lemma \ref{lemma:4.2} with $q = p + \frac{3}{2}$ and thus assume that
$\P\{X \geq x\} = O(x^{-p -\frac{3}{2}})$. Together with Lemma \ref{lemma:4.1}, this gives
\be \label{5.5}
\frac{n}{n^p}\ \E \xi_n(\gamma) \, = \, O\Big(\frac{1}{\sqrt{n}}\Big).
\en
Note that with this tail hypothesis, necessarily 
$\E X^\beta < \infty$ for any $\beta < p + \frac{3}{2}$. Since $p \geq 2$,
we have that the 3-rd moment $\E X^3$ is finite.

Applying both \eqref{5.4} and \eqref{5.5} in the representation \eqref{3.2} and using \eqref{5.3}, 
we thus obtain that
\be \label{5.6}
\E T_n^p \, = \, \E \xi_n 1_{B_{n,p}} + O\Big(\frac{1}{\sqrt{n}}\Big).
\en

It remains to study an asymptotic behavior of the last expectation in \eqref{5.6}.
On this step the argument is very similar to the one used in the proof
of Proposition \ref{proposition:1.2}. Note that 
$\frac{1}{n}\,S_n \geq \frac{1}{n}\,S_{n,p} \geq \frac{1}{2}$
on the set $B_{n,p}$ as long as $n \geq 2p$.
Applying the Taylor formula, we use an elementary inequality
$$
\Big|\frac{1}{x^p} - 1\Big| \leq p\,2^{p+1}\,|x - 1|, \qquad x \geq \frac{1}{2}.
$$
In particular, on the set $B_{n,p}$
$$
\Big|\frac{1}{(\frac{1}{n}\,S_n)^p} - 1\Big| \, \leq \, p\,2^{p+1}\,
\Big|\frac{1}{n}\,S_n - 1\Big|. 
$$
This gives
\bee
\big|\xi_n - X_1^2 \dots X_p^2\,\big|\, 1_{B_{n,p}}
 & \leq & 
p\,2^{p+1}\, X_1^2 \dots X_p^2\,
\Big|1 - \frac{1}{n}\,S_p - \frac{1}{n}\, S_{n,p}\Big| \\
 & \leq &
p\,2^{p+1}\, X_1^2 \dots X_p^2\,
\Big|1 - \frac{1}{n}\, S_{n,p}\Big| +
\frac{p\,2^{p+1}}{n}\, X_1^2 \dots X_p^2\,S_p,
\ene
so, taking the expected values,
\bee
\big|\E \xi_n1_{B_{n,p}} - \E X_1^2 \dots X_p^2\,1_{B_{n,p}}\big| 
 & \leq &
p\,2^{p+1}\,(\E X^2)^p\,\E\, \Big|1 - \frac{1}{n}\, S_{n,p}\Big| \\
 & & + \ 
\frac{p\,2^{p+1}}{n}\, (\E X^2)^{p-1}\,\E X^3.
\ene
In view of \eqref{5.2},
$$
\E X_1^2 \dots X_p^2\,1_{B_{n,p}} \, = \, \E X_1^2 \dots X_p^2 + e^{-cn}
 \, = \, (\E X^2)^p + o(e^{-cn}).
$$
for some constant $c>0$. Recalling \eqref{5.3}, we thus get that
\bee
\big|\E \xi_n - (\E X^2)^p\big| 
 & \leq &
p\,2^{p+1}\,(\E X^2)^p\,\E\, \Big|1 - \frac{1}{n}\, S_{n,p}\Big| \\
 & & + \ 
\frac{p\,2^{p+1}}{n}\, (\E X^2)^{p-1}\,\E X^3 + o(e^{-cn}).
\ene
Finally
\bee
\E\,\Big|\frac{1}{n}\, S_{n,p} - 1\Big|
 & = &
\frac{1}{n}\,\E\,|S_{n,p} - n| \, \leq \,
\frac{1}{n}\,\E\,|S_{n,p} - (n-p)| + \frac{p}{n} \\
 & \leq &
\frac{1}{n}\,\sqrt{\Var(S_{n,p})} + \frac{p}{n} \, \leq \,
\frac{1}{\sqrt{n}}\,\sqrt{\E X^2} + \frac{p}{n}.
\ene
It remains to refer to \eqref{5.6}.
}


We now consider the lemmas which enable us to get a bound for $\E R_n^{(p)},$ see \eqref{R_n}.
Without loss of generality, let $\E X = 1$ and $n \geq 2p$.

Introduce additional notation: $M_{n,r} = \max_{r < i \leq n} X_i, 
\ \ (1 \leq r \leq p).$

Recall that there is the representation \eqref{3.1} but instead of \eqref{3.2} we write now
\long\gdef\COMMENT#104{
$$
U_n^p \, = \, \sum X_{i_1}^2 \dots X_{i_p}^2
$$
with summation over all $i_1,\dots,i_p \in \{1,\dots,n\}$. For $r = 1,\dots,p$, 
we denoted by $\mathcal C(p,r)$ the collection  of all tuples
$\gamma = (\gamma_1,\dots,\gamma_r)$ of positive integers such that 
$\gamma_1 + \dots + \gamma_r = p$. For any $\gamma \in \mathcal C(p,r)$, 
there are $n(n-1)\dots (n-r+1)$ sequences $X_{i_1},\dots, X_{i_p}$ 
with $r$ distinct terms that are repeated $\gamma_1,\dots,\gamma_r$ times. 
Therefore, by the i.i.d. assumption,
}
\be \label{6.2}
\E R_n^{(p)} \ = \ \sum_{r=1}^p \frac{n(n-1)\dots (n-r+1)}{n^{p+1}}
\sum_{\gamma \in \mathcal C(p,r)} \E\, \eta_n(\gamma) M_n^2,
\en
where
$$
\eta_n(\gamma) = \frac{X_1^{2\gamma_1} \dots 
X_r^{2\gamma_r}}{(\frac{1}{n}\,S_r + \frac{1}{n}\, S_{n,r})^{p+1}}.
$$

In order to bound the expectations on the right-hand side of \eqref{6.2}, 
we use again the events $A_{n,r}$ and $B_{n,r}$, see \eqref{AB}.
From elementary inequalities $M_n \leq S_n$ and
$$
X_1^{2\gamma_1} \dots X_r^{2\gamma_r} \leq 
(X_1 + \dots + X_r)^{2\gamma_1 + \dots + 2\gamma_r} \leq S_n^{2p},
$$
it follows that
$
\eta_n(\gamma) M_n^2 \, \leq \, n^{p+1}\, S_n^{p+1} \, \leq \,
2^p n^{p+1} \, (S_r^{p+1} + S_{n,r}^{p+1}),
$
implying
\be \label{6.3}
\E\, \eta_n(\gamma)\,M_n^2\,1_{A_{n,r}} \leq
2^p n^{p+1}\,\Big(\E\,S_r^{p+1}\,\P(A_{n,r}) + 
\E S_r^{p+1}\, \E\,S_{n,r}^{p+1}\,1_{A_{n,r}}\Big).
\en
Here, by Lemma \ref{lemma:1.3} with $\lambda = 1/2$ and using $n-r \geq \frac{1}{2}\,n$, 
we have
\be \label{6.4}
\P(A_{n,r}) \leq \exp\Big\{-\frac{1}{16\,b^2}\,n\Big\}, \qquad b^2 = \E X^2.
\en
Also, assuming that the moment $\E X^{p+2}$ is finite and applying
the H\"older inequality with exponents ${(p+2)}/{(p+1)}$ and $p+2$, 
one may bound the last expectation in \eqref{6.3} as
$$
\E\,S_{n,r}^{p+1}\,1_{A_{n,r}} \leq 
\big(\E\,S_{n,r}^{p+2}\big)^{\frac{p+1}{p+2}}\,(\P(A_{n,r}))^{\frac{1}{p+2}}.
$$
By Jensen inequality,
$
\E\,S_{n,r}^{p+2} \leq r^{p+1}\,\E X^{p+2}.
$
Applying this in \eqref{6.3}, the inequality \eqref{6.4} yields an exponential bound
\be \label{6.5}
\E\, \eta_n(\gamma)\,M_n^2\,1_{A_{n,r}} \leq e^{-cn}
\en
with some constant $c>0$ which does not depend on $n$.

As for the set $B_{n,r}$, we use on it a point-wise upper bound \\
$
\eta_n(\gamma) \, \leq \, 2^{p+1}\, {X_1^{2\gamma_1} \dots 
X_r^{2\gamma_r}}/{(\frac{1}{n}\,S_r + 1)^{p+1}}.
$
One may also use 
$M_n \leq M_r + M_{n,r} \leq S_r + M_{n,r}$, implying, 
by Jensen's inequality,
$
M_n^2 \ \leq \ (r+1)\,(X_1^2 + \dots + X_r^2 + M_{n,r}^2).
$
It gives
\begin{eqnarray} \label{6.6}
\E\, \eta_n(\gamma)\,M_n^2\,1_{B_{n,r}} 
 & \leq &
2^{p+1}(r+1)\, \sum_{k=1}^r \E\,\frac{X_1^{2\gamma_1} \dots 
X_r^{2\gamma_r}}{(\frac{1}{n}\,S_r + 1)^{p+1}}\,X_k^2  \nonumber \\
 & & + \ 
2^{p+1}(r+1)\, \sum_{k=1}^r \E\,\frac{X_1^{2\gamma_1} \dots 
X_r^{2\gamma_r}}{(\frac{1}{n}\,S_r + 1)^{p+1}}\ \E M_{n,r}^2
\end{eqnarray}
Without an essential loss, the last expectation $\E M_{n,r}^2$ may be replaced 
with $\E M_n^2$. The second last expectation was considered in Lemmas \ref{lemma:4.2} -- \ref{lemma:4.3} 
under the condition \eqref{3.4}, which holds as long as the moment $\E X^{p+1}$ is 
finite. The third last expectation in \eqref{6.6}, due to an additional factor $X_k^2$,
dominates the second last and needs further consideration under stronger
moment assumptions. 
Recalling \eqref{6.5} and returning to \eqref{6.2}, let us summarize in the following
  statement.

\vskip5mm
\begin{lemma} \label{lemma:6.1}
 If the moment $\E X^{p+2}$ is finite, then 
\begin{eqnarray}\label{6.7}
c\,\E R_n^{(p)}
 & \leq & 
e^{-cn} + \max_{1 \leq k \leq r}\, \max_{\gamma \in \mathcal C(p,r)} \ 
\frac{1}{n^{p-r+1}}\
\E\,\frac{X_1^{2\gamma_1} \dots 
X_r^{2\gamma_r}}{(\frac{1}{n}\,S_r + 1)^{p+1}}\ X_k^2 \nonumber \\
 & & + \
\max_{\gamma \in \mathcal C(p,r)} \  \frac{1}{n^{p-r+1}}\
\E\,\frac{X_1^{2\gamma_1} \dots 
X_r^{2\gamma_r}}{(\frac{1}{n}\,S_r + 1)^{p+1}} \ \E M_n^2
\end{eqnarray}
with some constant $c>0$ which does not depend on $n \geq 2p$.
\end{lemma}
%

\vskip2mm
\noindent
In order to obtain polynomial bounds for the expectations in \eqref{6.7} under
suitable moment or tail assumptions, we need to develop corresponding analogs 
of Lemmas \ref{lemma:4.2} -- \ref{lemma:4.3}. We will consider separately the cases $r=1$,
$r = p$, and $2 \leq r \leq p-1$ under the tail condition
\be \label{7.1}
\P\{X \geq x\} = O(1/x^{p+\alpha}) \quad {\rm as} \ x \rightarrow \infty,
\en
where $\alpha > 0$ is a parameter. It implies that the moments 
$\E X^q$ are finite for all $q < p+\alpha$ and is fulfilled 
as long as the moment $\E X^{p+\alpha}$ is finite. Put
$
\ep(x) = x^{p+\alpha}\,(1-F(x)), 
$
where $F$ denotes the distribution function of the random variable $X$.

\begin{lemma} \label{lemma:7.1}
Under \eqref{7.1} with $1 < \alpha \leq p+2$,
\be \label{7.2}
\E\, \frac{X_1^{2p}}{(\frac{1}{n}\,X_1 + 1)^{p+1}} \, = \, 
O(n^{p - \alpha + 2}).
\en
Moreover, for any index $1 \leq k \leq r$,
\be \label{7.3}
\E\, \frac{X_1^{2p}}{(\frac{1}{n}\,X_1 + 1)^{p+1}}\ X_k^2 \, = \, 
O(n^{p - \alpha + 2} \log n).
\en
\end{lemma}
\begin{proof}
The expectation in \eqref{7.2} is equal to and satisfies
\bee
\E\, \frac{X_1^{2p}}{(\frac{1}{n}\,X_1 + 1)^{p+1}}\,1_{\{X_1 \geq n\}} +
\E\, \frac{X_1^{2p}}{(\frac{1}{n}\,X_1 + 1)^{p+1}}\,1_{\{X_1 < n\}}
 & & \\
 & & \hskip-50mm \leq \
n^{p+1}\, \E X^{p-1}\,1_{\{X \geq n\}} + \E\, X^{2p}\,1_{\{X < n\}}  
\ene
Similarly to \eqref{by_part}, we get
\bee
\E\, X^{2p}\,1_{\{X < n\}} 
\leq 
2p \int_0^n x^{p-\alpha-1}\,\ep(x)\,dx \ = \ O(n^{p-\alpha}),
\ene
provided that $\alpha < p$. In the case $\alpha = p$, the last integral is bounded
by $O(\log n)$. In addition,
\bee
\E\, X^p\,1_{\{X \geq n\}}
= 
O(n^{p-\alpha}) + p \int_n^\infty x^{p-\alpha-1}\,\ep_q(x)\,dx \ = \ O(n^{p-\alpha}).
\ene
This proves \eqref{7.2} for $\alpha \leq p$. If $p < \alpha \leq p+2$, then \eqref{7.2} holds automatically, since then $2p < p+\alpha$ and therefore the expectation in \eqref{7.2}
does not exceed the finite moment $\E X_1^{2p}$, while the right-hand side
is bounded away from zero.

For the second assertion, one may assume that $k=1$, in which case the 
expectation in \eqref{7.3} is equal to and satisfies
\bee
\E\, \frac{X_1^{2p+2}}{(\frac{1}{n}\,X_1 + 1)^{p+1}} 
 & = &
\E\, \frac{X_1^{2p+2}}{(\frac{1}{n}\,X_1 + 1)^{p+1}}\,1_{\{X_1 \geq n\}} +
\E\, \frac{X_1^{2p+2}}{(\frac{1}{n}\,X_1 + 1)^{p+1}}\,1_{\{X_1 < n\}} \\
 & \leq &
n^{p+1}\, \E X^{p+1}\,1_{\{X \geq n\}} + \E\, X^{2p+2}\,1_{\{X < n\}}.
\ene
Here, similarly to the previous step, if $\alpha < p+2$,
$$
\E\, X^{2p+2}\,1_{\{X < n\}} \, \leq \,
2p \int_0^n x^{p-\alpha+1}\,\ep(x)\,dx \, = \, O(n^{p-\alpha+2}).
$$
In the case $\alpha = p+2$, the last integral 
is bounded by $O(\log n)$. In addition,
\bee
\E\, X^{p+1}\,1_{\{X \geq n\}}
=
O(n^{-\alpha + 1}) + p \int_n^\infty x^{-\alpha}\,\ep(x)\,dx \ = \ O(n^{-\alpha + 1}).
\ene
\end{proof}
%
\begin{lemma} \label{lemma:7.2}
If the moment $\E X^4$ is finite, then 
$$
\E\, \frac{X_1^2 \dots X_p^2}{(\frac{1}{n}\,S_p + 1)^{p+1}} \, = \, 
O(1).
$$
Moreover, for any index $1 \leq k \leq p$,
$$
\E\, \frac{X_1^2 \dots X_p^2}{(\frac{1}{n}\,S_p + 1)^{p+1}}\ X_k^2 \, = \, 
O(1).
$$
\end{lemma}
\vskip2mm
This statement is clear. The last expectation does not exceed
$\E X^4\, (\E X^2)^{p-1}$ which is finite and does not depend on $n$.

\begin{lemma} \label{lemma:7.4}
Let $\gamma = (\gamma_1,\dots,\gamma_r) \in \mathcal C(p,r)$,
$2 \leq r \leq p$. Under the condition \eqref{7.1} with
$2 < \alpha \leq 4$, for any index $1 \leq k \leq r$,
\be \label{7.5}
\E\,\frac{X_1^{2\gamma_1} \dots X_r^{2\gamma_r}}{(\frac{1}{n}\,S_r + 1)^{p+1}}\,
X_k^2 \, = \, O(n^{p-r-\alpha + 4}).
\en
\end{lemma}
\vskip5mm
\begin{proof}
One may reformulate \eqref{7.5} as the statement
\be \label{7.6}
\E\,\frac{X_1^{2\gamma_1} \dots X_r^{2\gamma_r}}{(\frac{1}{n}\,S_r + 1)^{p+1}} 
\, = \, O(n^{p-r-\alpha+4})
\en
in which $\gamma = (\gamma_1,\dots,\gamma_r) \in \mathcal C(p+2,r)$.
If all $\gamma_i \leq \frac{p+2}{2}$, there is nothing to prove,
since then
$$
\E\,\frac{X_1^{2\gamma_1} \dots X_r^{2\gamma_r}}{(\frac{1}{n}\,S_r + 1)^{p+1}}
\, \leq \, \E\,X_1^{2\gamma_1} \dots \E\,X_r^{2\gamma_r} \, \leq \,
(\E X^{p+2})^r.
$$

In the other case, we repeat the arguments used in the proof 
of Lemma \ref{lemma:4.3}. Suppose for definiteness that $\gamma_1$ is the largest number 
among all $\gamma_i$'s. Necessarily $\gamma_1 > \frac{p+2}{2}$ and therefore 
$\gamma_i < \frac{p+2}{2}$ 
for all $i \geq 2$. Since $S_r < n \Rightarrow X_1 < n$, we similarly have
\be \label{7.7}
\E\,\frac{X_1^{2\gamma_1} \dots X_r^{2\gamma_r}}{(\frac{1}{n}\,S_r + 1)^{p+1}}\
1_{\{S_r < n\}} \, \leq \, (\E X^{p+2})^{r-1}\,
\E\,X^{2\gamma_1}\,1_{\{X < n\}}.
\en

To bound the last expectation, note that, if $x = n$ is the point of 
continuity of $F(x)$,
\bee
\E\, X^{2\gamma_1}\,1_{\{X < n\}}
 & = &
- n^{2\gamma_1}\,(1 - F(n)) + 
2\gamma_1 \int_0^n x^{2\gamma_1-1}\,(1-F(x))\,dx \\
 & \leq &
2\gamma_1 \int_0^n x^{2\gamma_1 - 1- p - \alpha}\,\ep(x)\,dx.
\ene
But since $\gamma_1 \leq p-r+3$ (which follows from 
$\gamma_1 + \gamma_2 + \dots + \gamma_r = p+2$ and $\gamma_i \geq 1$), we have
$$
\int_1^n x^{2\gamma_1 - p - \alpha - 1}\,\ep_{p+\alpha}(x)\,dx \, \leq \,
\int_1^n x^{p - 2r - \alpha + 5}\,\ep(x)\,dx.
$$
The last integral grows at the desired rate $O(n^{p-r-\alpha + 4})$ as the worst
case, if and only if $p - 2r - \alpha + 5 \leq p-r-\alpha + 3$, that is,
$r \geq 2$ (which is true). Thus,
$$
\E\,X^{2\gamma_1}\,1_{\{X < n\}} \, = \, O(n^{p-r-\alpha+4})
$$
In view of \eqref{7.7}, this proves \eqref{7.6} for the part of the expectation restricted 
to the set $S_r<n$, that is,
\be \label{7.8}
\E\,\frac{X_1^{2\gamma_1} \dots X_r^{2\gamma_r}}{(\frac{1}{n}\,S_r + 1)^{p+1}}\ 
1_{\{S_r < n\}} \, = \, O(n^{p-r-\alpha + 4}).
\en
Here, the worst situation is attained in the case $r=2$,
$\gamma_1 = p+1$, $\gamma_2 = 1$.

Turning to the expectation over the complementary set $S_r \geq n$,
introduce the events 
$$
\Omega_i = \Big\{X_i \geq \max_{j \neq i} X_j\Big\}, \quad i = 1,\dots,r.
$$
On every such set, $X_i \leq S_r \leq rX_i$. In particular,
$S_r \geq n$ implies $X_i \geq n/r$. Hence, together with \eqref{7.8}, 
\eqref{7.6} would follow from the inequality
\be \label{7.9}
\E\,\frac{X_1^{2\gamma_1} \dots X_r^{2\gamma_r}}{X_i^{p+1}}\,
1_{\{X_i \geq n\} \cap \Omega_i} \, = \, O(n^{-r-\alpha+3})
\en
with an arbitrary index $1 \leq i \leq r$.

{\it Case 1.}
$i \geq 2$. If we fix any values $X_1 = x_1$ and $X_i = x_i$,
then the expectation with respect to $X_j$, $j \neq i$, in \eqref{7.9} will yield 
a bounded quantity (since the $(p+2)$-moment is finite). Hence \eqref{7.9} 
is simplified to
\be \label{7.10}
\E\,X_1^{2\gamma_1} X_i^{2\gamma_i - p - 1}\,
1_{\{X_i \geq n\} \cap \{X_i \geq X_1\}} \, = \, O(n^{-r-\alpha+3}).
\en
Here, the expectation over $X_1$ may be estimated similarly to 
the previous step, by replacing $n$ with $x_i$. Recall that
$\gamma_1 > \frac{p+2}{2}$ and hence $2\gamma_1 \geq p+3$. 

{\it Case 1.1.}
$2\gamma_1 > p + \alpha$. Then we have
$$
\E\,X_1^{2\gamma_1} 1_{\{X_1 \leq x_i\}} \, \leq \, 
2\gamma_1 \int_0^{x_i} x^{2\gamma_1 - 1 - p - \alpha}\,\ep(x)\,dx \, \leq \, 
Cx_i^{2\gamma_1 - p - \alpha}
$$
with some constant $C>0$. Hence, up to a constant, the expectation in \eqref{7.9} is 
bounded by
\bee
\E\,X_i^{2\gamma_i + 2\gamma_1 - 2p - \alpha - 1}\,1_{\{X_i \geq n\}}
 & = &
\int_n^\infty x^{2\gamma_i + 2\gamma_1 - 2p - \alpha - 1}\,dF(x) \\
 & & \hskip-45mm = \ 
n^{2\gamma_i + 2\gamma_1 - 2p - \alpha - 1}\,(1-F(n)) +
c_i \int_n^\infty x^{2\gamma_i + 2\gamma_1 - 2p - \alpha - 2}\,(1-F(x))\,dx \\
 & & \hskip-45mm = \ 
O(n^{2\gamma_i + 2\gamma_1 - 3p - 2\alpha - 1}) +
c_i\int_n^\infty x^{2\gamma_i + 2\gamma_1 - 3p - 2\alpha - 2}\,\ep(x)\,dx \\
 & & \hskip-45mm = \ 
O(n^{2\gamma_i + 2\gamma_1 - 3p - 2\alpha - 1}).
\ene
To obtain \eqref{7.10}, it remains to check that
$2\gamma_i + 2\gamma_1 - 3p - 2\alpha - 1 \leq -r-\alpha+3$.
And indeed, since
\be \label{7.11}
p = \gamma_i + \gamma_1 + \sum_{j \neq i,1} \gamma_j \geq \gamma_i + \gamma_1 + (r-2),\nonumber
\en
the desired relation would follow from
$2\,(p - (r-2)) - 3p - 2\alpha - 1 \leq -r-\alpha+3$, that is, $p+r \geq \alpha$ 
(which is true since $\alpha \leq 4$, while $p,r \geq 2$).
\vskip2mm
{\it Case 1.2.}
  $2\gamma_1 \leq p + \alpha$. Then 
$\E\,X_1^{2\gamma_1} 1_{\{X_1 \leq x_i\}} \leq \E\,X^{p+\alpha}$ which is bounded
in $x_i$, and the expectation in \eqref{7.9} does not exceed up to a constant
\bee
\E\,X_i^{2\gamma_i + 2\gamma_1 - p - 1}\ 1_{\{X_i \geq n\}}
 & = &
\int_n^\infty x^{2\gamma_i + 2\gamma_1 - p - 1}\,dF(x) \\
 & & \hskip-45mm = \ 
n^{2\gamma_i + 2\gamma_1 - p - 1}\,(1-F(n)) +
c_i \int_n^\infty x^{2\gamma_i + 2\gamma_1 - p - 2}\,(1-F(x))\,dx \\
 & & \hskip-45mm = \ 
O(n^{2\gamma_i + 2\gamma_1 - 2p - \alpha - 1}) +
c_i  \int_n^\infty 
x^{2\gamma_i + 2\gamma_1 - 2p - \alpha - 2}\,\ep(x)\,dx \ = \ 
O(n^{2\gamma_i + 2\gamma_1 - 2p - \alpha - 1}).
\ene
To obtain \eqref{7.10}, it remains to check that
$$
2\gamma_i + 2\gamma_1 - 2p - \alpha - 1 \leq -r-\alpha+3. 
$$
And indeed, by \eqref{7.10}, the desired relation would follow from
$2\,(p - (r-2)) - 2p - \alpha - 1 \leq -r-\alpha+3$, that is, $r \geq 0$.

{\it Case 2.}
 $i = 1$. If we fix any value $X_1 = x_1$ in \eqref{7.9},
the expectation with respect to $X_j$, $j \neq 1$, yields a bounded
quantity. Hence \eqref{7.9} is simplified to
$$
\E\,X^{2\gamma_1 - p - 1}\, 1_{\{X \geq n\}} = O(n^{-r-\alpha+3}).
$$
We have
\bee
\E\,X^{2\gamma_1 - p - 1}\, 1_{\{X \geq n\}} 
 & = &
\int_n^\infty x^{2\gamma_1 - p - 1}\,dF(x) \\
 & & \hskip-20mm = \
n^{2\gamma_1 - p - 1}\,(1 - F(n)) + (2\gamma_1 - p - 1)
\int_n^\infty x^{2\gamma_1 - p - 2}\,(1-F(x))\,dx \\
 & & \hskip-20mm = \
O(n^{2\gamma_1 - 2p- \alpha - 1}) +
\int_n^\infty x^{2\gamma_1 - 2p - \alpha - 2}\,\ep(x)\,dx \ = \
O(n^{2\gamma_1 - 2p - \alpha - 1}).
\ene
It remains to see that $2\gamma_1 - 2p - \alpha - 1 \leq -r-\alpha + 3$,
that is, $2\gamma_1 + r \leq 2p + 4$. But this follows from
$p+2 = \gamma_1 + \dots + \gamma_r \geq \gamma_1 + (r-1) \geq \gamma_1 + \frac{r}{2}$.
\end{proof} 

\long\gdef\COMMENT#108{
%
%
\noindent
Let us apply these lemmas in the inequality \eqref{6.7}. Using the bounds
for the cases $r=1$, $r = p$, and $2 \leq r \leq p-1$, and assuming that
\eqref{7.1} is fulfilled for an integer $p \geq 2$ and 
a real number $2 < \alpha \leq 4$, they imply that 
$$
c\,\E R_n^{(p)} \, \leq \,
e^{-cn} + \Big(\frac{1}{n^{\alpha - 2}} + \frac{1}{n} + 
\frac{1}{n^{\alpha-3}}\Big) +
\Big(\frac{\log n}{n^{\alpha - 2}} + \frac{1}{n} + \frac{\log n}{n^2}\Big) \ \E M_n^2,
$$
where the constant $c>0$ does not depend on $n$.
To simplify, we have to assume that $\alpha \geq 3$ leading to
\be \label{8.1}
c\,\E R_n \, \leq \, \frac{1}{n^{\alpha-3}} + \frac{\log n}{n}\, \E M_n^2.
\en

The last expectation in \eqref{8.1} may also be estimated
in a polynomial way. Namely, since for any $q \geq 2$,
$$
M_n^2 \, \leq \, (X_1^2 + \dots + X_n^q)^{\frac{2}{q}},
$$
we get, by Jensen's inequality,
$$
\E M_n^2 \, \leq \, 
(\E X_1^q + \dots + \E X_n^q)^{\frac{2}{q}} \, = \, 
n^{\frac{2}{q}}\,(\E X^q)^{\frac{2}{q}}.
$$
Therefore, choosing $2 < q < p+\alpha$ to be sufficiently close to $p+\alpha$,
and using $\alpha = 7/2$, from \eqref{8.1} we obtain the following statement. Recall that 
the functional $R_n^{(p)}$ was defined in \eqref{R_n}.

\vskip5mm
\begin{proposition}\label{propos:8.1}

Given an integer $p \geq 2$, suppose that
$$
\P\{X \geq x\} = O(1/x^{p+7/2}) \quad {\rm as} \ x \rightarrow \infty.
$$
then
$$
\E R_n^{(p)} \, = \, O\Big(\frac{1}{\sqrt{n}}\Big).
$$
\end{proposition}
\vskip2mm
As another example, when $\E X^{p+4}$ is finite, the rate can be improved to
$\E R_n^{(p)} = O(n^{-\frac{p-2}{p + 4}})$. Moreover, if
$\E\, e^{\ep X} < \infty$ for some $\ep > 0$, then
\be \label{8.2}
\E R_n^{(p)} = O\Big(\frac{(\log n)^2}{n}\Big).
\en
Indeed, the finiteness of the exponential moment of $X$
is actually equivalent to the family of moment bounds
$$
(\E X^q)^{1/q} \leq cq, \qquad q \geq 1,
$$
which for $q \geq 2$ give
\bee
\E M_n^2 
 & \leq & 
\E\,(X_1^q + \dots + X_n^q)^{2/q} \\
 & \leq & 
(\E X_1^q + \dots + \E X_n^q)^{2/q} \ = \ n^{2/q}\, (\E X^q)^{2/q} \ \leq \
(cq)^2\, n^{2/q}.
\ene
Choosing here $q$ to be of order $\log n$, we arrive at
$$
\E M_n^2 \leq C\,(\log n)^2
$$
with a constant $C$ independent of $n$. Applying this bound in \eqref{8.1}
with $\alpha=4$, we then obtain the much better rate as in \eqref{8.2}.
}

\section{Proofs of main results}
\label{sec:4}
\begin{proof} {\it{of Theorem \ref{theorem1}}}.
Fix $k\ge 3.$ In the following we omit $k$ in notation when it is not necessary. So we write $S_{n}=S_{n}(k),$ $\lambda=\lambda(k),$ $Z=Z_k,$ $I=I(k).$ We say that a random variable $V$ has a mixed Poisson distribution with mixing distribution $F$  when, for every integer $m\ge 0,$
\[
\P(V=m)=\E \left(e^{-\Lambda}\frac{\Lambda^{m}}{m!}\right),
\]
where $\Lambda$ is a random variable with distribution $F.$

Put $\Lambda=\sum_{\alpha \in I}\E_{W_{1}, W_{2},...,W_{n}}Y_{\alpha}.$

We have for any real function $h:$ $\{0,1,2,...\} \rightarrow \R$
\be \label{100}
|\E h(S_{n})-\E h(Z)|\le \E|\E_{W_{1},...,W_{n}}h(S_{n})-\E_{W_{1},...,W_{n}}h(V)| + |\E h(V)- \E h(Z)|.
\en
For each $\alpha \in I$, define $B_{\alpha}\equiv \{ \beta \in I:$ $\alpha$ and $\beta$ have at least one edge in common$\}$.
Put 
\[
b_{1}=\sum_{\alpha \in I}\sum_{\beta\in B_{\alpha}}p_{\alpha}p_{\beta},
\]
where $p_{\alpha}=\E_{W_{1},...,W_{n}}Y_{\alpha}.$
\[
b_{2}=\sum_{\alpha \in I}\sum_{\alpha\neq\beta\in B_{\alpha}}p_{\alpha\beta},
\]
where $p_{\alpha\beta}=\E_{W_{1},...,W_{n}}Y_{\alpha}Y_{\beta}.$

Note that, for any $\alpha\in I$ and $\beta\in I \setminus B_{\alpha},$ the cycles $\alpha$ and $\beta$ may have joint vertices but the do not have any edge in common. Therefore, for such $\alpha$ and $\beta$, the random variable $Y_{\alpha}$ and $Y_{\beta}$ are conditionally independent given weights $W_{1},...,W_{n}.$ Thus, by Theorem 1 in \cite{Goldstein}, proved with the Chen--Stein method, and relations \eqref{1} and \eqref{100}, we get
\be \label{101}
\parallel \mathscr{L}(S_{n}(k))-\mathscr{L}(Z) \parallel\, \lesssim \E(b_{1}+b_{2})+|\E\Lambda-\lambda(k) |,
\en
where we write here and in the following that $A_n\, \lesssim B_n$ or $A_n\, \gtrsim B_n$ when there exists a positive constant $c$ not depending on $n$ such that $A_n\, \leq c B_n$ or $A_n\, \geq c B_n.$

For random variables $b_1$ and $b_2$, we get, cf. \eqref{3.2},
 by the i.i.d. assumption and simple inequality for positive $c$ and $d: 2cd \leq c^2 + d^2$,
\be \label{b_1}
\E (b_1 + b_2) \ \lesssim \ \sum_{p=k+2}^{2k} \sum_{r=k}^{p-1} \frac{n(n-1)\dots (n-r+1)}{n^p}
\sum_{\gamma \in \mathcal C(p,r)} \E \psi_n(\gamma),
\en
where
$$ \label{psi}
\psi_n(\gamma) = {W_1^{2\gamma_1} \dots 
W_r^{2\gamma_r}}/{(\frac{1}{n}\,L_r + \frac{1}{n}\, L_{n,r})^p}
$$
and 
$$
L_r = W_1 + \dots + W_r, \quad L_{n,r} = W_{r+1} + \dots + W_n.
$$
For example, we have minimal value $p=k+2$ and $r=k+1$ for the cycles $\alpha = (1,2,\dots,k)$ and $\beta = (1,2,\dots,k-1,k+1).$ Then
$$
\E p_{\alpha\beta}\lesssim \E W_1^4 W_2^2  \dots  W_{k-1}^2 W_k^2 W_{k+1}^2/L_n^{k+2}.
$$
We have maximal value $p=2k$ and $r=k$ for the cycle $\alpha = (1,2,\dots,k)$. Then
$$
\E p_{\alpha}^2 \leq \E W_1^4 \dots W_k^4/L_n^{2k}.
$$
Lemmas \ref{lemma:4.1} and \ref{lemma:4.3} and inequality \eqref{b_1} under condition \eqref{asumpW} imply
\be \label{bb}
\E (b_1 + b_2) = o\left(\frac{\log n}{n}\right).
\en
Now we construct an upper bound for the last summand in \eqref{101}.

It is clear that 
\be \label{upper}
\E\Lambda \leq \frac{1}{2k}\E \left(\frac{(W_1^2 + \dots +  W_n^2)^k}{L_n^k}\right).
\en

On the other hand, note that for a positive $a$ and positive sequence $\{x_i\}, 
i = 1,2,\dots k,$ we have, see e.g. Lemma 8 in \cite{Liu_moment}
$$
\prod_{i=1}^k\frac{1}{a+x_i}\geq \frac{1}{a^k} - 
\frac{\sum_{i=1}^kx_i}{a^{k+1}}.
$$
Therefore, by the i.i.d. assumption, we get
\begin{eqnarray}
\label{lower}
\E\Lambda &\geq &\frac{1}{2k}\E \left(\frac{(W_1^2 + \dots +  W_n^2)^k}{L_n^k}\right) - c_1 \sum_{r=1}^{k-1} \frac{n(n-1)\dots (n-r+1)}{n^k}
\sum_{\gamma \in \mathcal C(k,r)} \E \psi_n(\gamma)\nonumber\\
&&- \, 
\frac{c_2}{n}\sum_{\gamma \in \mathcal C(k+1,k)} \E \left(\frac{W_1^{2\gamma_1} \dots 
W_r^{2\gamma_r}}{(L_n/n)^{k+1}}\right),
\end{eqnarray} 
where $c_1$ and $c_2$ do not depend on $n$.

Combining Lemmas \ref{lemma:4.1} and \ref{lemma:4.3} and relations \eqref{1.2}, \eqref{101}, \eqref{bb}, \eqref{upper} and \eqref{lower}, we finish the proof of Theorem \ref{theorem1}. 
\end{proof}
\begin{proof}  {\it{of Theorem \ref{proposition:1.2}}}
We split the proof of the Theorem into several steps. Without loss of generality, let $\E X = 1$.

{\it Necessity.} By Lemma \ref{lemma:3.2}, for the convergence $\E T_n^p \rightarrow (\E X^2)^p$, 
it is necessary that all summands in \eqref{3.2} with $r < p$ should be vanishing 
at infinity. In particular, for the shortest tuple $\gamma$ with $r=1$
as in Lemma \ref{lemma:3.1}, it should be required that
${n^{1-p}}\,\E\,\xi_n(\gamma) \rightarrow 0$ as $n \rightarrow \infty$.
Hence, from the inequality \eqref{short} it follows that
\be  \label{3.6}
\E X^p\, 1_{\{X \geq n\}} = o(1/n).\nonumber
\en
This relation may be simplified in terms of the tails
of the distribution of $X$. Indeed,
$
\E\,X^p\,1_{\{X \geq n\}} \, \geq \, n^p\,\P\{X \geq n\},
$
so that the property \eqref{3.4}
is necessary for the convergence 
 $\E T_n^p \rightarrow (\E X^2)^p$.

\long\gdef\COMMENT#105{
For the particular collection $\gamma = (p)$ with $r=1$ we have, see \eqref{xi},
\bee
\E \xi_n(\gamma) 
 & \geq &
\E_{X_1}\,\frac{X_1^{2p}}{(\frac{1}{n}\,X_1 + \frac{1}{n}\,\E_{S_{n,1}}\, S_{n,1})^p}\\
 & = &
\E\,\frac{X^{2p}}{(\frac{1}{n}\,X + \frac{n-1}{n})^p} \ \geq \
2^{-p}\,n^p \ \E X^p\, 1_{\{X \geq n\}},
\ene
where we applied Jensen's inequality.
In view of \eqref{4.2}, for the boundedness of the sequence $\E T_n^p$
it is therefore necessary that
\be \label{3.3}
\E X^p\, 1_{\{X \geq n\}} = o(1/n)
\en
as $n \rightarrow \infty$. In particular, the moment $\E X^p$ has to be finite.

The relation \eqref{3.3} may be simplified in terms of the tails
of the distribution of $X$. Indeed,
$$
\E\,X^p\,1_{\{X \geq n\}} \, \geq \, n^p\,\P\{X \geq n\},
$$
\long\gdef\COMMENT#102{
so that the property \eqref{3.4}
is necessary for \eqref{3.3}. On the other hand, from \eqref{3.4} it follows that
$\ep_{p+1}(x) \rightarrow 0$ as $x \rightarrow \infty$.
Hence, assuming without loss of generality that $x = n$ is the point of 
continuity of $F$, we get
\begin{eqnarray} \label{tail}
\E\, X^p\,1_{\{X > n\}}
 & = &
\int_n^\infty x^p\,dF(x) \ = \
n^p\,(1 - F(n)) + p\,\int_n^\infty x^{p-1}\,(1-F(x))\,dx 
\nonumber\\
 & = &
o(1/n) + p\,\int_n^\infty \frac{\ep(x)}{x^2}\,dx \\ 
& = &
o(1/n) + 
p\,\sup_{x > n} \, \ep(x)\,\frac{1}{n} \ = \ o(1/n).\nonumber
\end{eqnarray}
}
Thus, the tail condition \eqref{3.4} is necessary for the convergence 
$\E T_n^p \rightarrow (\E X^2)^p$ as stated in Theorem \ref{proposition:1.2}
(and actually for the boundedness of the $p$-th moments of $T_n$).
}
{\it Sufficiency and rate of convergence.}
First note that the condition \eqref{3.4} ensures that the moment $\E X^p$ is finite. 
For the convergence part of Theorem \ref{proposition:1.2} 
we apply Lemmas \ref{lemma:4.1}-\ref{lemma:4.3}, 
which imply that $\E \xi_n(\gamma) = o(n^{p-r})$ for any collection
$\gamma = (\gamma_1,\dots,\gamma_r)$ with $r<p$. It remains to take into account
Lemma \ref{lemma:3.2} about the longest tuple $\tilde \gamma = (1,\dots,1)$ of length $r=p$
and to recall the representation \eqref{3.2}.

\long\gdef\COMMENT#107{
For the convergence part, we apply Lemmas \ref{lemma:4.1}-\ref{lemma:4.3}, which imply that 
$\E \xi_n(\gamma) = o(n^{p-r})$ for any collection
$\gamma = (\gamma_1,\dots,\gamma_r)$ with $r<p$. It remains to to take into account Lemma \ref{lemma:3.2} about
the longest tuple $\tilde \gamma = (1,\dots,1)$ of length $r=p$, in which case
$$
\xi_n \equiv \xi_n(\tilde \gamma) \, = \, \frac{X_1^2 \dots 
X_p^2}{(\frac{1}{n}\,S_n)^p}.
$$
According to the representation \eqref{3.2}, we thus obtain the characterization
\be \label{5.1}
\E T_n^p \rightarrow (\E X^2)^p \ \Longleftrightarrow \ \E \xi_n = 
\E\,\frac{X_1^2 \dots X_p^2}{(\frac{1}{n}\,S_n)^p} \rightarrow
(\E X^2)^p.
\en

Let us use the same notations as before.
We have \eqref{5.2} and \eqref{4.1}.
On the set $B_{n,p}$ 
there is pointwise upper bound 
$$
\xi_n \, = \, \frac{X_1^2 \dots 
X_p^2}{(\frac{1}{n}\,S_p + \frac{1}{n}\, S_{n,p})^p} \, \leq \,
2^p X_1^2 \dots X_p^2.
$$
Hence, the random variables $\xi_n\,1_{B_{n,p}}$ have an an integrable 
majorant. Since also $\xi_n \rightarrow X_1^2 \dots X_p^2$ 
(the law of large numbers) and $1_{B_{n,p}} \rightarrow 1$ a.s. (implied by \eqref{5.2}), 
one may apply the Lebesgue dominated convergence theorem, which gives 
$$
\E \xi_n 1_{B_n} \rightarrow (\E X^2)^p.
$$
Together with \eqref{4.1}, we get $\E \xi_n \rightarrow (\E X^2)^p$,
and \eqref{5.1} is proved.

{\it Rate of convergence.} 
}
Turning to the rate of convergence, first note that by
Lemma \ref{lemma:4.1} and \ref{lemma:4.3}, for any $\gamma \in \mathcal C(p,r)$ with $2 \leq r \leq p-1$,
\be \label{5.4}
\frac{n(n-1)\dots (n-r+1)}{n^p}\ \E \xi_n(\gamma) \, = \, 
o\Big(\frac{\log n}{n}\Big)
\en
For the shortest tuple $\gamma = (p)$ with $r=1$, we apply
Lemma \ref{lemma:4.2} with $q = p + \frac{3}{2}$ and thus assume that
$\P\{X \geq x\} = O(x^{-p -\frac{3}{2}})$. Together with Lemma \ref{lemma:4.1}, this gives
\be \label{5.5}
\frac{n}{n^p}\ \E \xi_n(\gamma) \, = \, O\Big(\frac{1}{\sqrt{n}}\Big).
\en
Note that with this tail hypothesis, necessarily 
$\E X^\beta < \infty$ for any $\beta < p + \frac{3}{2}$. Since $p \geq 2$,
we have that the 3-rd moment $\E X^3$ is finite.
 Applying both \eqref{5.4} and \eqref{5.5} in the representation \eqref{3.2} and using \eqref{4.1}, 
we thus obtain that
\be \label{5.6}
\E T_n^p \, = \, \E \xi_n 1_{B_{n,p}} + O({1}/{\sqrt{n}}),
\en
with $\xi_n$ defined in \eqref{xi_n}.

It remains to study an asymptotic behavior of the last expectation in \eqref{5.6}.
Note that 
$\frac{1}{n}\,S_n \geq \frac{1}{n}\,S_{n,p} \geq \frac{1}{2}$
on the set $B_{n,p}$ as long as $n \geq 2p$.
Applying the Taylor formula, we use an elementary inequality
$
|{x^{-p}} - 1| \leq p\,2^{p+1}\,|x - 1|$ for $  x \geq \frac{1}{2}.
$
In particular, on the set $B_{n,p}$ one has 
$
\Big|{(\frac{1}{n}\,S_n)^{-p}} - 1\Big| \, \leq \, p\,2^{p+1}\,
\Big|\frac{1}{n}\,S_n - 1\Big|. 
$
This gives
\bee
\big|\xi_n - X_1^2 \dots X_p^2\,\big|\, 1_{B_{n,p}}
 & \leq & 
p\,2^{p+1}\, X_1^2 \dots X_p^2\,
\Big|1 - \frac{1}{n}\,S_p - \frac{1}{n}\, S_{n,p}\Big| \\
 & \leq &
p\,2^{p+1}\, X_1^2 \dots X_p^2\,
\Big|1 - \frac{1}{n}\, S_{n,p}\Big| +
\frac{p\,2^{p+1}}{n}\, X_1^2 \dots X_p^2\,S_p,
\ene
so, taking the expected values,
\bee
\big|\E \xi_n1_{B_{n,p}} - \E X_1^2 \dots X_p^2\,1_{B_{n,p}}\big| 
  &\leq &
p\,2^{p+1}\,(\E X^2)^p\,\E\, \Big|1 - \frac{1}{n}\, S_{n,p}\Big| \\
 & & 
 + \ 
\frac{p\,2^{p+1}}{n}\, (\E X^2)^{p-1}\,\E X^3.
\ene
In view of \eqref{5.2},
$$
\E X_1^2 \dots X_p^2\,1_{B_{n,p}} \, = \, \E X_1^2 \dots X_p^2 + e^{-cn}
 \, = \, (\E X^2)^p + o(e^{-cn}).
$$
for some constant $c>0$. Recalling \eqref{4.1}, we thus get that
\bee
\big|\E \xi_n - (\E X^2)^p\big| 
 & \leq &
p\,2^{p+1}\,(\E X^2)^p\,\E\, \Big|1 - \frac{1}{n}\, S_{n,p}\Big| \\
 & & + \ 
\frac{p\,2^{p+1}}{n}\, (\E X^2)^{p-1}\,\E X^3 + o(e^{-cn}).
\ene
Finally
\bee
\E\,\Big|\frac{1}{n}\, S_{n,p} - 1\Big|
 & = &
\frac{1}{n}\,\E\,|S_{n,p} - n| \, \leq \,
\frac{1}{n}\,\E\,|S_{n,p} - (n-p)| + \frac{p}{n} \\
 & \leq &
\frac{1}{n}\,\sqrt{\Var(S_{n,p})} + \frac{p}{n} \, \leq \,
\frac{1}{\sqrt{n}}\,\sqrt{\E X^2} + \frac{p}{n}.
\ene
It remains to refer to \eqref{5.6}.
\end{proof}
\begin{proof}  {\it{of Theorem \ref{proposition:3}}}.
\noindent
Let us apply Lemmas \ref{lemma:7.1}--\ref{lemma:7.4} in the inequality \eqref{6.7}. Using the bounds
for the cases $r=1$, $r = p$, and $2 \leq r \leq p-1$, and assuming that
\eqref{7.1} is fulfilled for an integer $p \geq 2$ and 
a real number $2 < \alpha \leq 4$, they imply that 
$$
\E R_n^{(p)} \, \leq \,
e^{-cn} + \Big(\frac{1}{n^{\alpha - 2}} + \frac{1}{n} + 
\frac{1}{n^{\alpha-3}}\Big) +
\Big(\frac{\log n}{n^{\alpha - 2}} + \frac{1}{n} + \frac{\log n}{n^2}\Big) \ \E M_n^2,
$$
where the constant $c>0$ does not depend on $n$.
To simplify, we have to assume that $\alpha \geq 3$ leading to
\be \label{8.1}
c\,\E R_n \, \leq \, \frac{1}{n^{\alpha-3}} + \frac{\log n}{n}\, \E M_n^2.
\en

The last expectation in \eqref{8.1} may also be estimated
in a polynomial way. Namely, since for any $q \geq 2$, one has 
$
M_n^2 \, \leq \, (X_1^2 + \dots + X_n^q)^{{2}/{q}},
$
we get, by Jensen's inequality,
$$
\E M_n^2 \, \leq \, 
(\E X_1^q + \dots + \E X_n^q)^{\frac{2}{q}} \, = \, 
n^{\frac{2}{q}}\,(\E X^q)^{\frac{2}{q}}.
$$
Therefore, choosing $2 < q < p+\alpha$ to be sufficiently close to $p+\alpha$,
and using $\alpha = 7/2$, from \eqref{8.1} we obtain \eqref{th-3}.
\long\gdef\COMMENT#109{
the following statement. Recall that 
the functional $R_n^{(p)}$ was defined in \eqref{R_n}.

\vskip5mm
\begin{proposition}\label{propos:8.1}

Given an integer $p \geq 2$, suppose that
$$
\P\{X \geq x\} = O(1/x^{p+7/2}) \quad {\rm as} \ x \rightarrow \infty.
$$
then
$$
\E R_n^{(p)} \, = \, O\Big(\frac{1}{\sqrt{n}}\Big).
$$
\end{proposition}
\vskip2mm
As another example, 
}
When $\E X^{p+4}$ is finite, we get \eqref{th-31}.

At last, to prove \eqref{8.2}, note that 
the finiteness of the exponential moment of $X$
is actually equivalent to the family of moment bounds
$
(\E X^q)^{1/q} \leq cq, $ for $ q \geq 1,
$
which for $q \geq 2$ give
\bee
\E M_n^2 
  \leq  
\E\,(X_1^q + \dots + X_n^q)^{2/q} 
  \leq  
(\E X_1^q + \dots + \E X_n^q)^{2/q} \ 
\leq \
(cq)^2\, n^{2/q}.
\ene
Choosing here $q$ to be of order $\log n$, we arrive at
$
\E M_n^2 \leq C\,(\log n)^2
$
with a constant $C$ independent of $n$. Applying this bound in \eqref{8.1}
with $\alpha=4$, we then obtain the much better rate as in \eqref{8.2}.
\end{proof}
\begin{acknowledgement}
This research was done within the framework of the Moscow Center for Fundamental and Applied Mathematics, Lomonosov Moscow State University, and HSE University Basic Research Programs. Theorem 1 was proved under support
of the RSF grant No. 18-11-00132.  Research of S. Bobkov was partially supported by NSF grant DMS--1855575.
\end{acknowledgement}
%
\input{references}

\end{document}

%% file: references.tex
%
%
%

%% file: BUD_ArXiv.bbl
\begin{thebibliography}{99.}%

\bibitem{Goldstein89} R. Arratia, L. Goldstein, and L. Gordon, Two moments suffice for Poisson approximations: the Chen-Stein method. Ann. Probab. \textbf{17}, 9 -- 25 (1989).

\bibitem{Goldstein} R. Arratia, L. Goldstein, and L. Gordon, Poisson approximation and the Chen-Stein method. Statistical Science \textbf{5}, 403 -- 434 (1990).


\bibitem{Bhamidi} S. Bhamidi, R. van der Hofstad, J.S.H. van Leeuwaarden, Novel scaling limits for critical inhomogeneous random graphs. Ann. Probab. \textbf{40}, 2299--2361 (2012).

\bibitem{Bollobas} B. Bollob?as, S. Janson, O. Riordan, The phase transition in inhomogeneous random graphs. Random Struct. Algorithms  \textbf{31}, 3 -- 122 (2007).

\bibitem{Britton} T. Britton, M. Deijfen, A. Martin-L?of, Generating simple random graphs with prescribed degree distribution. J. Stat. Phys. \textbf{124}, 1377 -- 1397 (2006).

\bibitem{Chakrabarti} D. Chakrabarti  et al., Epidemic thresholds in real networks. //ACM Transactions on Information and System Security (TISSEC). \textbf{10}(4), 1 -- 26 (2008).



\bibitem{Chung2} F. Chung and L. Lu, Connected components in random graphs with given expected degree sequences. Ann. Combinat. \textbf{6}(2), 125 -- 145 (2002).

\bibitem{Chung} F. Chung, L. Lu, The volume of the giant component of a random graph with given expected degrees. SIAM J. Discrete Math. \textbf{20}, 395 -- 411 (2006)
(electronic).

\bibitem{Demarco} Demarco, B., Kahn, J., 2012. Upper tails for triangles. Random Structures Algorithms,  \textbf{40}(4), 452--459.

\bibitem{Erdos} P. Erdos, A. Renyi, On Random Graphs. Publ. Math. Debrecen \textbf{6}, 290 -- 297 (1959).

\bibitem{Faloutsos} C. Faloutsos, P. Faloutsos, and M. Faloutsos, On power-law relationships of the internet topology. Computer Communications Rev., \textbf{29}, 251--262 (1999).

\bibitem{Gilbert} E.N. Gilbert, Random graphs. Annals of Mathematical Statistics \textbf{30}, 1141 -- 1144 (1959).


\bibitem{Hofstad} R. van der Hofstad, Random Graphs and Complex Networks. Cambridge Series in Statistical and Probabilistic Mathematics, \textbf{1} (2017).


\bibitem{Ulyanov} Z. Hu, V. Ulyanov, Q. Feng, Limit theorems for number of edges in the generalized random graphs with random vertex weights. Journal of Mathematical Sciences, \textbf{218}(2), 231 -- 237 (2016).

\bibitem{Janson2} S. Janson, K. Oleszkiewicz, A. Rucinski, Upper tails for subgraph counts in random graphs. Israel J. Math. \textbf{142}, 61 -- 92 (2004).

\bibitem{Janson} S. Janson, Asymptotic equivalence and contiguity of some random graphs, Random Struct. Algorithms \textbf{36}, 26 -- 45 (2010).


\bibitem{Kim} J.H Kim,  V.H. Vu, Divide and conquer martingales and the number of triangles in a random graph. Random Structures Algorithms \textbf{24}(2), 166 -- 174 (2004).

\bibitem{Liu2020} Q. Liu, Z. Dong, Limit laws for the number of triangles in the generalized random graphs with random node weights. Statistics $\And$ Probability Letters \textbf{161}, 108733 (2020).

\bibitem{Liu_moment}Q. Liu, Z. Dong, Moment-based spectral analysis of large-scale generalized random graphs. IEEE Access \textbf{5}, 9453 -- 9463 (2017).

\bibitem{Norros} I. Norros, H.Reittu, On a conditionally Poisson graph process. Adv. in Appl. Probab. \textbf{38}, 59 -- 75 (2006).

\bibitem{Preciado} V. M. Preciado and A. Jadbabaie, Moment-based spectral analysis of large-scale networks using local structural information. // IEEE/ACM Trans. Netw \textbf{21}(2), 373 -- 382 (2013).


%
\end{thebibliography}
